\documentclass[12pt]{amsart}

\usepackage[top=0.9in, bottom=1in, left=1in, right=1in]{geometry}
\usepackage{amsmath, amssymb, amsthm}
\usepackage{verbatim}
\usepackage{graphicx}
\usepackage{enumerate}
\usepackage{mathrsfs}
\usepackage{hyperref}
\usepackage{makecell}

\newtheorem{thm}{Theorem}[section]

\newtheorem{lem}[thm]{Lemma}
\newtheorem{cor}[thm]{Corollary}

\def\XXint#1#2#3{{\setbox0=\hbox{$#1{#2#3}{\int}$ }
\vcenter{\hbox{$#2#3$ }}\kern-.6\wd0}}

\theoremstyle{definition}
\newtheorem{definition}[thm]{Definition}

\theoremstyle{remark}

\newtheorem{remark}[thm]{Remark}

\numberwithin{equation}{section}

\newtheoremstyle{ser}
{8pt}
{8pt}
{\it}
{}
{\sf}
{:}
{6mm}
{}

\newtheoremstyle{serr}
{8pt}
{8pt}
{\normalfont}
{}
{\sf}
{.}
{6mm}
{}

\theoremstyle{ser}

\theoremstyle{serr}

\begin{document}


\title{An Operator Inequality for Range Projections}


\author{Soumyashant Nayak}

\address{Smilow Center for Translational Research, University of Pennsylvania, Philadelphia, PA 19104}
\email{nsoum@upenn.edu}
\urladdr{https://nsoum.github.io} 




\begin{abstract}
By a result of Lundquist-Barrett it follows that the rank of a positive semi-definite matrix is less than or equal to the sum of the ranks of its principal diagonal submatrices when written in block form. In this article, we take a general operator algebraic approach which provides insight as to why the above rank inequality resembles the Hadamard-Fischer determinant inequality in form, with multiplication replaced by addition. It also helps in identifying the necessary and sufficient conditions under which equality holds. Let $\mathscr{R}$ be a von Neumann algebra, and $\Phi$ be a normal conditional expectation from $\mathscr{R}$ onto a von Neumann subalgebra $\mathscr{S}$ of $\mathscr{R}$. Let $\mathfrak{R}[T]$ denote the range projection of an operator $T$. For a positive operator $A$ in $\mathscr{R}$, we prove that $\Phi(\mathfrak{R}[A]) \le \mathfrak{R}[\Phi(A)]$ with equality if and only if $\mathfrak{R}[A] \in \mathscr{S}$. 
\end{abstract}






\maketitle

\section{Introduction}

Let $\langle n \rangle := \{1, 2, \cdots, n \}$. For a matrix $A \in M_n(\mathbb{C})$ and non-empty indexing sets $\alpha, \beta \subseteq \langle n \rangle $, we denote by $A[\alpha, \beta ]$ the submatrix of $A$ obtained by considering the intersection of rows indicated by $\alpha$ and columns indicated by $\beta$. The principal diagonal block $A[\alpha, \alpha]$ is denoted by $A[\alpha]$. Let $\alpha ^c$ be the complement of $\alpha$ in $\langle n \rangle$. The Hadamard-Fischer determinant inequality states that 
\begin{equation}
\det (A) \le \det (A[\alpha]) \cdot \det (A[\alpha ^c]),
\end{equation}
with equality if and only if $A = A[\alpha] \oplus A[\alpha ^c].$ Extensions of the above determinant inequality to finite von Neumann algebras were studied in \cite{gen-hadamard} by the author with special attention to the equality conditions. In \cite[Theorem 1]{rank}, Lundquist and Barrett prove an analogous inequality involving the ranks of $A, A[\alpha], A[\alpha ^c]$ with multiplication replaced by addition as follows 
\begin{equation}
\mathrm{rank}(A) \le \mathrm{rank}(A[\alpha]) + \mathrm{rank}(A[\alpha ^c]).
\end{equation}

In this article, we strive to show that this analogy is not merely a coincidence but rather follows naturally from a more general inequality. Furthermore, we identify the necessary and sufficient conditions under which equality holds whereas these were not considered in \cite{rank}. This seems to be a common limitation in some of the literature on matrix-based inequalities where it is a formidable task to keep track of the equality conditions. The author is of the opinion that an algebro-analytic approach makes it more readily apparent how the equality conditions arise.

We set up some notation for the rest of the article. We denote by $\mathrm{tr}$, the usual trace of a matrix in $M_n(\mathbb{C})$, which yields the sum of the diagonal entries of the matrix. For us, $\mathscr{H}$ denotes a complex Hilbert space and $\mathcal{B}(\mathscr{H})$ denotes the set of bounded operators on $\mathscr{H}$.
 
\begin{definition}
Let $T$ be a bounded operator in $\mathcal{B}(\mathscr{H})$. The {\it range} of $T$ is the set $\{ Tx : x \in \mathscr{H} \} \subseteq \mathscr{H}$. The (orthogonal) projection onto the closure of the range of $T$ is called the {\it range projection} of $T$, and denoted by $\mathfrak{R}[T]$.
\end{definition}

\begin{remark}
\label{rmrk:ran_proj}
Note that $\mathfrak{R}[T] T = T$. In fact, $\mathfrak{R}[T]$ is the smallest projection $E$ in $\mathcal{B}(\mathscr{H})$ satisfying $ET=T$.
\end{remark}

A von Neumann algebra is a weak-operator closed self-adjoint subalgebra of $\mathcal{B}(\mathscr{H})$ containing the identity operator. Let $\mathscr{R}$ be a von Neumann algebra and $\mathscr{S}$ be a von Neumann subalgebra of $\mathscr{R}$. We recall that a conditional expectation $\Phi$ from $\mathscr{R}$ onto $\mathscr{S}$ is a unital idempotent $\mathscr{S}$-bimodule map where $\mathscr{R}$ is considered as a $\mathscr{S}$-bimodule via the usual left and right multiplication. The reader may wish to take a quick glance at the relevant subsections of \cite[\S 2]{gen-hadamard} for a quick refresher in the theory of von Neumann algebras, conditional expectations, operator monotone functions which are relevant to the present discussion. For von Neumann algebras $\mathscr{R}_1, \mathscr{R}_2$, a positive linear map $\Psi : \mathscr{R}_1 \mapsto \mathscr{R}_2$ is said to be {\it normal} if $\Psi(\sup_{\alpha \in I} H_{\alpha}) = \sup_{\alpha \in I} \Psi(H_{\alpha})$ for any increasing net of self-adjoint operators $H_{\alpha}$. We state the main result of this article below.

\begin{thm}
\label{thm:main}\textsl{Let $\mathscr{R}$ be a von Neumann algebra and $\mathscr{S}$ be a von Neumann subalgebra of $\mathscr{R}$. For a positive operator $A$ in $\mathscr{R}$ and a normal conditional expectation $\Phi$ from $\mathscr{R}$ onto $\mathscr{S}$, we have
\begin{equation}
\label{eqn:rank_main}
\Phi(\mathfrak{R}[A]) \le \mathfrak{R}[\Phi(A)]
\end{equation}
with equality if and only if $\mathfrak{R}[A] \in \mathscr{S}$.
}
\end{thm}

For readers who are primarily interested in applications to matrix theory, it may be helpful to use the dictionary below to mentally substitute terminology for von Neumann algebras with the corresponding ones for the special case of finite-dimensional self-adjoint matrix algebras. 
\begin{table}[ht]
\centering
\begin{tabular}{|c|c|}
\hline
$\mathscr{H}$ & $\mathbb{C}^n$ \\
\hline $\mathscr{R}$ & $M_n(\mathbb{C})$ \\
\hline $\mathscr{S}$ &  \makecell{$M_{n_1}(\mathbb{C}) \oplus \cdots \oplus M_{n_k}(\mathbb{C}),$  \\ $ \sum_{i=1}^k n_i = n$} \\
\hline $\Phi$ & \makecell{$A \mapsto \mathrm{diag}(A_{11}, \cdots, A_{kk}),$ \\ $A_{ii} \in M_{n_i}(\mathbb{C}), 1\le i \le k$}\\
\hline
\end{tabular}
\label{tab:partial}
\end{table}

Further, one should keep in mind that the various topologies on a von Neumann algebra such as the norm topology, weak-operator topology, strong-operator topology are all equivalent on $M_n(\mathbb{C})$. In the proof of Theorem \ref{thm:main}, we use the normality of the conditional expectation $\Phi$ in a crucial manner but that assumption is superfluous in the context of $M_n(\mathbb{C})$.

\section{Proof of the Main Theorem}

We first compile observations and results that are necessary for our proof of Theorem \ref{thm:main}. There are various operator versions of Jensen's inequality in the literature (\cite{hansen_pedersen1}, \cite{hansen_pedersen2}) involving operator monotone functions and completely positive maps in the context of $C^*$-algebras. We note the relevant inequality below without making the effort to state it in its most general form.

\begin{thm}[Jensen's operator inequality]
\label{thm:opmon}
\textsl{Let $\mathfrak{A}$ be a $C^*$-algebra and $\Psi : \mathfrak{A} \mapsto \mathfrak{A}$ be a unital completely positive map. Let $f$ be an operator monotone function on the interval $[0, \infty)$. For a positive operator in $A$ in $\mathfrak{A}$, we have $$\Psi(f(A)) \le f(\Psi(A)).$$}
\end{thm}
Note that the necessary and sufficient conditions for equality are not obviously apparent without further assumptions on $\Psi$.

Using the integral representation of operator monotone functions (see \cite{lowner}), and noting that
$$t^r = \frac{\sin r \pi }{ \pi } \int_{0}^{\infty} \frac{(1 + \lambda ) t}{\lambda + t} \frac{\lambda ^{r-1 } }{1 + \lambda} \; d\lambda, r \in (0, 1)$$
we have that $t^r$ is an operator monotone function on $[0, \infty)$ for $0 < r < 1$. As a corollary of the preceding theorem, we have the following result.

\begin{cor}
\label{cor:power}
\textsl{
Let $\mathfrak{A}$ be a $C^*$-algebra and $\Psi : \mathfrak{A} \mapsto \mathfrak{A}$ be a completely positive map. Let $f$ be an operator monotone function on the interval $[0, \infty)$ and $r \in (0, 1)$. For a positive operator in $A$ in $\mathfrak{A}$, we have
\begin{equation}
\label{eqn:main}
\Psi(A^r) \le \Psi(A)^r.
\end{equation}
}
\end{cor}

The final piece we need is \cite[Lemma 5.1.5]{kadison-ringrose} which we paraphrase below.
\begin{lem}
\label{lem:lim_range}
\textsl{
For a positive operator $A$ in $\mathcal{B}(\mathscr{H})$ such that $A \le I$, we have $$ \lim_{r \rightarrow 0} A^r = \mathfrak{R}[A]$$ in the strong-operator topology.
}
\end{lem}
The assumption that $A \le I$ in the above lemma may be relaxed by noting that for a positive operator $A$, the operator $\frac{1}{\|A\|}A \le I, \lim_{r \rightarrow 0} \|A\|^r = 1$, and $\mathfrak{R}[A] = \mathfrak{R}[(\frac{1}{\|A\|}A)]$. Thus for any positive operator $A$ in $\mathcal{B}(\mathscr{H})$, we observe that in the strong-operator topology, $\lim_{r \rightarrow 0} A^r$ is the range projection of $A$. Furthermore as von Neumann algebras are closed in the strong-operator topology, if $A \in \mathscr{R}$, then $\mathfrak{R}[A]$ is also in $\mathscr{R}$.

\begin{lem}
\label{lem:main}
\textsl{Let $\mathscr{R}$ be a von Neumann algebra and $\mathscr{S}$ be a von Neumann subalgebra of $\mathscr{R}$. For a projection $E$ in $\mathscr{R}$ and a conditional expectation $\Phi$  from $\mathscr{R}$ onto $\mathscr{S}$,  $\Phi(E)$ is a projection if and only if $\Phi(E) = E$.}
\end{lem}
\begin{proof}
If $\Phi(E)$ is a projection, then we have that $\Phi(E)^2 = \Phi(E) = \Phi(E^2)$. Applying $\Phi$ to the positive operator $(E - \Phi(E))^2$, we see that $\Phi((E- \Phi(E))^2) = \Phi(E^2 - E\Phi(E) - \Phi(E)E + \Phi(E)^2) = \Phi(E)^2 - \Phi(E^2) = 0$. From the faithfulness of $\Phi$, we conclude that $(E - \Phi(E))^2 = 0$ which implies that $\Phi(E) = E$. The converse is straightforward.
\end{proof}

\begin{proof}[Proof of Theorem \ref{thm:main}]
By \cite[Theorem 1.5.10]{brown_ozawa}, $\Phi$ is a completely positive map. From inequality (\ref{eqn:main}), we have that $\Phi(A^r) \le \Phi(A)^r$ for $0 < r < 1$. Using the normality of $\Phi$ and taking limits in the strong-operator topology as $r \rightarrow 0$, we observe that $\Phi(\mathfrak{R}[A]) = \Phi(\lim_{r \rightarrow 0} A^r) = \lim_{r \rightarrow 0} \Phi(A^r) \le \lim_{r \rightarrow 0} \Phi(A)^r = \mathfrak{R}[\Phi(A)]$.  Next we prove the necessity and sufficiency of the equality condition. As $\mathfrak{R}[A], \mathfrak{R}[\Phi(A)]$ are projections in $\mathscr{R}$, using Lemma \ref{lem:main} we see that $\Phi(\mathfrak{R}[A]) = \mathfrak{R}[\Phi(A)] \Rightarrow \Phi(\mathfrak{R}[A]) = \mathfrak{R}[A] \Rightarrow \mathfrak{R}[A] \in \mathscr{S}$. Conversely, if $\mathfrak{R}[A] \in \mathscr{S}$, as $\mathfrak{R}[A] A = A$ and $\Phi$ is a $\mathscr{S}$-bimodule map,  we have $\Phi(\mathfrak{R}[A]A) = \Phi(A) \Rightarrow \mathfrak{R}[A] \Phi(A) = \Phi(A) $. In the vein of Remark \ref{rmrk:ran_proj}, for any operator $T \in \mathscr{R}$, there is a smallest projection $E$ in $\mathscr{R}$ such that $ET = T$ and it is given by $\mathfrak{R}[T]$. We conclude that $\mathfrak{R}[\Phi(A)] \le \mathfrak{R}[A] = \Phi(\mathfrak{R}[A])$ and thus $\Phi(\mathfrak{R}[A]) =  \mathfrak{R}[\Phi(A)].$ 
\end{proof}

\section{Applications}

In this section, we apply Theorem \ref{thm:main} in the setting of a {\it finite} von Neumann algebra $\mathscr{R}$. Let $\tau$ denote the unique faithful and normal center-valued trace on $\mathscr{R}$. 

\begin{remark}
The column rank of an orthogonal projection in $M_n(\mathbb{C})$ is equal to its (unnormalized) trace. This follows from the fact that the range of an orthogonal projection is the eigenspace associated with the eigenvalue $1$ and the trace is the multiplicity of the eigenvalue $1$. Thus the rank of a matrix in $M_n(\mathbb{C})$ is equal to the trace of its range projection. 
\end{remark}

The above observation forms the basis of the definition of the rank of an operator $T$ in $\mathscr{R}$ as noted below.

\begin{definition}
For an operator $T$ in $\mathscr{R}$, we call $\tau(\mathfrak{R}[T])$,  which is a positive operator in the center $\mathscr{Z}$, the {\it center-valued rank} or simply {\it rank} of $T$ and denote it by $\mathfrak{r}(T)$. 
\end{definition}
The center-valued rank was first introduced in \cite{tr_rank} as tr-rank and studied in further detail in \cite{nayak_rank_id} by the author. In the case of factors, as the center is isomorphic to $\mathbb{C}$, we abuse notation and consider $\mathfrak{r}$ as a positive real-valued function. Note that the usual rank of a matrix $A$ in $M_n(\mathbb{C})$ is $n \cdot \mathfrak{r}(A)$. 

\begin{cor}
\label{cor:cv_rank}
\textsl{ Let $\mathscr{S}$ be a von Neumann subalgebra of $\mathscr{R}$ and $\Phi : \mathscr{R} \mapsto \mathscr{S}$ be a $\tau$-preserving normal conditional expectation onto $\mathscr{S}$. For a positive operator $A$ in $\mathscr{R}$, we have the following inequality,
\begin{equation*}
\mathfrak{r}(A) \le \mathfrak{r}(\Phi(A))
\end{equation*}
with equality if and only if $\mathfrak{R}[A]$ is in $\mathscr{S}$.
}
\end{cor}
\begin{proof}
As $\Phi$ is $\tau$-preserving, by inequality (\ref{eqn:rank_main}) we have that $\mathfrak{r}(A) = \tau(\mathfrak{R}[A]) = \tau(\Phi(\mathfrak{R}[A]))\le \tau(\mathfrak{R}[\Phi(A)]) = \mathfrak{r}(\Phi(A))$. Using the faithfulness of $\tau$, we conclude that  $\mathfrak{r}(A) = \mathfrak{r}(\Phi(A)) \Leftrightarrow \Phi(\mathfrak{R}[A]) = \mathfrak{R}[\Phi(A)] \Leftrightarrow \mathfrak{R}[A] \in \mathscr{S}$.
\end{proof}

\begin{cor}
\textsl{For a partition of $\langle n \rangle $ into indexing sets $\alpha_1, \cdots, \alpha_k$, we have the following inequality,
$$\mathrm{rank}(A) \le \sum_{i=1}^k \mathrm{rank}(A[\alpha_i])$$
with equality if and only if  $\mathfrak{R}[A] \in M_{n_1}(\mathbb{C}) \oplus \cdots \oplus M_{n_k}
(\mathbb{C})$ {\it i.e.}\ $\mathfrak{R}[A][\alpha_{\ell}, \alpha_m] = [0 \cdots ]$ for $\ell \ne m, 1 \le \ell, m \le k.$}
\end{cor}
\begin{proof}
Let $n_i := \# (\alpha_i), 1 \le i \le k$ and let $\tau := \frac{\mathrm{tr(\cdot)}}{n}$ be the unique tracial state on $M_n(\mathbb{C})$. Recall that the usual rank is a scaling of $\mathfrak{r}$ by $n$. The map $\Phi : M_n(\mathbb{C}) \mapsto M_{n_1}(\mathbb{C}) \oplus \cdots \oplus M_{n_k}
(\mathbb{C})$ given by $\Phi(A) = \oplus_{i=1}^k A[\alpha_i]$ is a $\tau$-preserving conditional expectation from $M_n(\mathbb{C})$ onto $M_{n_1}(\mathbb{C}) \oplus \cdots \oplus M_{n_k}(\mathbb{C})$. As $\mathrm{rank}(\oplus_{i=1}^k A[\alpha_i]) = \sum_{i=1}^k \mathrm{rank}(A[\alpha_i])$, we obtain the required result using Corollary \ref{cor:cv_rank}.
\end{proof}

Note that for a matrix in $M_n(\mathbb{R})$, the rank does not depend on whether we consider it as an element of $M_n(\mathbb{R})$ or $M_n(\mathbb{C})$ since a collection of vectors in $\mathbb{R}^n \subset \mathbb{C}^n$ are $\mathbb{R}$-linearly independent if and only if they are $\mathbb{C}$-linearly independent. Further from Lemma \ref{lem:lim_range}, it is clear that if $A$ has real entries, then $\mathfrak{R}[A]$ has real entries. Thus the above result also holds for positive semi-definite matrices in $M_n(\mathbb{R})$. 

\begin{remark}
In Theorem \ref{thm:opmon}, using the operator monotone function {\it log} yields the Hadamard-Fischer determinant inequality as noted in \cite{gen-hadamard}, whereas in this article using the operator monotone functions $t^r, 0 < r < 1$ (and a limiting procedure) we have obtained the rank inequality in \cite[Theorem 1 (6)]{rank}. This explains their similar form.

An anonymous reviewer has kindly pointed out that these inequalities follow from the classical concavity of the functions {\it log}, $t \mapsto t^r (0 \le r \le 1)$ and the use of the operator monotone property (which coincides with concavity in this situation) is superfluous. The author agrees with the former observation but is more open-minded about the latter observation. This stems from a search for the appropriate unification of the Hadamard-Fischer-Koteljanskii determinant inequality and the other analogous Lundquist-Barrett rank inequalities via commuting conditional expectations (which are not communicated in this short paper and part of future work).    
\end{remark}

\end{document}